\documentclass[letter,12pt]{amsart}
\usepackage{amsmath, amsthm, amstext}
\usepackage{amssymb, array, amsfonts}
\usepackage[english]{babel}
\usepackage{bbding}
\usepackage{float}
\usepackage{enumerate}
\usepackage{tikz}
\usepackage{pdfpages}
\usepackage{hyperref}
\usepackage{float}
\usepackage{graphics}
\usepackage{mathtools}
\usepackage{graphicx}
   \usepackage{enumerate}
\usepackage[margin=2.2cm]{geometry}
\numberwithin{equation}{section}
\usepackage{tikz}
\usetikzlibrary{shapes}
\usepackage{xparse}
\usetikzlibrary{calc,arrows.meta,positioning}
\def\cont{\mathtt{cont}}
\def\totalcont{\mathtt{totalcont}}
\def\q{\mathtt{quasi}}
\def\quasi{\mathtt{quasi}}
\def\mult{\mathtt{mult}}
\def\foot{\mathtt{foot}}

\def \ep{\varepsilon}

\def \F {\mathcal F}
\def \A {\mathcal A}
\def \cP {\mathcal P}
\def \cR {\mathcal R}
\def \cQ {\mathcal Q}

\renewcommand{\l}{\left}

\renewcommand{\r}{\right}

\def \rL{\mathrm{L}}

\def \C{\mathbb{C}}

\def \N{\mathbb{N}}
\def \M2{\mathrm{M}_2}
\def \R{\mathbb{R}}
\def \Z{\mathbb{Z}}

\def \T{\mathbb{T}}
\def \A{\mathcal{A}}

\def \sl2r{\mathrm{SL}(2,\R)}

\newcommand{\beq}{\begin{equation}}
\newcommand{\eeq}{\end{equation}}

\def\diag{\operatorname{diag}}

\newcommand{\eqdef}{\stackrel{\rm def}{=\kern-3.6pt=}}

\theoremstyle{plain}
\newtheorem{theorem}{\bf Theorem}[section]
\newtheorem{lemma}[theorem]{\bf Lemma}

\newtheorem{prop}[theorem]{\bf Proposition}
\newtheorem{cor}[theorem]{\bf Corollary}

\theoremstyle{definition}

\theoremstyle{remark}
\newtheorem{remark}[theorem]{\bf Remark}

\theoremstyle{cond}

\renewcommand{\le}{\leqslant}
\renewcommand{\ge}{\geqslant}

\renewcommand{\qed}{\vrule height7pt width5pt depth0pt}
\title{Absence of flat bands for discrete periodic graph operators with generic potentials}
\author[M. Faust]{Matthew Faust}
\address{Department of Mathematics,
	Michigan State University,
	Wells Hall, 619 Red Cedar Road,
	East Lansing, MI, 48840,
	United States of America}
\thanks{The first named author was partially supported by the National Science Foundation DMS--2052519 grant.}

\author[I. Kachkovskiy]{Ilya Kachkovskiy}
\address{Department of Mathematics,
	Michigan State University,
	Wells Hall, 619 Red Cedar Road,
	East Lansing, MI, 48840,
	United States of America}
\thanks{The second named author was partially supported by the National Science Foundation DMS--1846114 and  DMS--2052519 grants, and the Sloan Research Fellowship 2022.}
\begin{document}
\maketitle
\begin{abstract}
We show that Schr\"odinger-type operators on discrete connected periodic graphs do not have flat bands for generic potentials.
\end{abstract}
\section{Introduction}
In this paper, we consider discrete long range Schr\"odinger-type operators on periodic graphs. Under some natural finiteness conditions, spectral theory of such operators can be studied using the Floquet transform. In the self-adjoint case, the spectra of these operators consist of finite unions of intervals (spectral bands). In many situations, such as the usual discrete periodic Schr\"odinger operator on $\mathbb Z^d$, it is known that the bands are non-flat (in other words, no band collapses into a point), in which case the spectrum of the corresponding operator is purely absolutely continuous. For more general graphs it is known that flat bands (or, equivalently, eigenvalues in the spectrum that will automatically have infinite multiplicity) exist in some cases. The most simple way to produce a flat band is to consider a periodic graph with a finite connected component. Clearly, each periodic copy of this component will carry a set of eigenvectors with compact supports. However, examples of flat bands can also be found beyond this (trivial) mechanism, see, for example, \cite{Sabri,SK,ktw,Overview}.

The goal of the present paper is to show that, for all discrete periodic graphs satisfying some natural finiteness and connectedness conditions, {\it an open dense set of potentials does not produce any flat bands}. 
\subsection{Schr\"odinger-type operators on discrete periodic graphs} Let $(\mathcal V,\tau,\mathfrak v,\mathfrak a)$ be the following data describing the discrete a $\Z^d$-periodic graph and the additional structures required to define a Schr\"odinger-type operator.
\begin{itemize}
    \item[(p1)] $\mathcal V$ is an infinite set with a free action $\tau\colon \mathbb Z^d\times \mathcal V\to \mathcal V$ of $\Z^d$. For $x\in \mathcal V$, this action will be denoted by $x\mapsto x+\alpha$, $\alpha\in \Z^d$, and will be sometimes referred to as a translation of $x$ by $\alpha$.
    \item[(p2)] The potential $\mathfrak v\colon \mathcal V\to \mathbb \C$ and the edge weight function $\mathfrak a\colon \mathcal V\times \mathcal V\to \C$ are translation-invariant (in other words, $\Z^d$-periodic):
    $$
    \mathfrak v(x+\alpha)=\mathfrak v(x),\quad \mathfrak a(x,y)=a(x+\alpha,y+\alpha),\quad \forall x,y\in \mathcal{V},\,\alpha\in \Z^d.
    $$
    \item[(p3)] The edge weight function satisfies the weak symmetry condition
    $$
\mathfrak a(x,y)=0\quad\text{if and only if }\quad \mathfrak a(y,x)=0;\quad \mathfrak a(x,x)=0, \quad \forall x,y\in \mathcal{V}.
    $$
    One can define the adjacency relation by $x\sim y$ if and only if $\mathfrak a(x,y)\neq 0$. As a consequence, the pair $\Gamma=(\mathcal V,\sim)$ becomes a (weighted) graph with an action of $\Z^d$.
    \item[(p4)] For the above graph $\Gamma$, we assume that the degree of each vertex is finite, and the fundamental domain of $\Gamma$ with respect to the action of $\mathbb Z^d$ contains only finitely many vertices and finitely many edges. The latter condition is equivalent to the fact that, for each $x\in \mathcal V$, there are only finitely many $\alpha\in \Z^d$ with $\mathfrak a(x,x+\alpha)\neq 0$.    
\end{itemize}
We note that Condition (p4) is a manifestation of local finiteness, which is assumed to hold in most of the current work on discrete and quantum graphs. Without this condition, spectral theory of periodic Schr\"odinger-type operators is much more complicated, as one can see from the recent work \cite{Sabri2}. The weak symmetry condition in (p3) is important for many of the constructions in this paper. We note that one usually assumes $\mathfrak a$ to be symmetric or self-adjoint as in \eqref{eq_self_adjoint} below.

For $\psi\in \ell^2(\mathcal V)$,  the Schr\"odinger-type operator $H$ on the above graph is defined as follows:
\beq
\label{eq_h_def}
(H\psi)(x)=\mathfrak v(x)\psi(x)+\sum\limits_{y\in \mathcal V\colon y\sim x}\mathfrak a(x,y)\psi(y).
\eeq
Due to (p4), the latter sum is always finite. Assuming the above conditions, it is easy to see that \eqref{eq_h_def} defines a bounded operator on $\ell^2(\mathcal V)$, which will be self-adjoint if and only if 
\beq
\label{eq_self_adjoint}
\mathfrak v(x)\in \R,\quad \mathfrak a(x,y)=\overline{\mathfrak a(y,x)}, \quad \forall x,y\in \mathcal{V}.
\eeq
Moreover, this operator is $\Z^d$-periodic, in the sense that it commutes with every $\Z^d$-translation.

\subsection{Floquet theory and flat bands} In order to state the main result and provide context for references, we will also need introduce some basic results of Floquet theory. Let $\Lambda=\mathcal V/\Z^d$ be a fundamental domain of $\mathcal V$ under the action of $\Z^d$. We will identify $\Lambda$ with a subset of $\mathcal{V}$ by picking some representative from each equivalence class. Let 
$$
N:=\#\Lambda<+\infty
$$
as was assumed in (p4) above. The Floquet transform will depend on the specific choice of the representative set, but recalculating between different such choices is straightforward. Define the new edge weight function $\mathfrak b\colon \Lambda\times\Lambda\to \C[z_1,z_1^{-1},\ldots,z_d,z_d^{-1}]$ by
\beq
\label{eq_b_def}
\mathfrak b(x,y)=\mathfrak b(z;x,y):=\sum\limits_{\alpha\in \Z^d}\mathfrak a(x,y+\alpha)z^{\alpha}=\sum_{\alpha\in \mathcal A}b_{\alpha}(x,y)z^{\alpha},\quad z^{\alpha}=z_1^{\alpha_1}\cdot\ldots\cdot z_d^{\alpha_d}.
\eeq
Note that the first summation, in fact, happens over the finite set $\{\alpha\in \Z^d\colon x\sim y+\alpha\}$, since otherwise $\mathfrak a(x,y+\alpha)=0$. As $\Lambda$ itself is finite, the notation in the second summation reflects the fact that the total range of values of the multi-index $\alpha$ can be assumed to be contained in some finite subset $\A\subset\Z^d$, independent of $x$.

For $z\in (\C^*)^d=(\C\setminus\{0\})^d$, one can now introduce the following family of ``fiber operators''
\beq
\label{eq_hz_def}
(h(z)\psi)(x)=\mathfrak v(x)\psi(x)+\sum\limits_{y\in \Lambda}\mathfrak b(z;x,y)\psi(y),\quad \psi\in \ell^2(\Lambda).
\eeq
Clearly, the above is an analytic (in fact, algebraic) family of operators acting on the finite-dimensional Hilbert space $\ell^2(\Lambda)$. Since the latter can be identified with $\C^N$, one can also consider them as $(N\times N)$-matrix-valued functions. From the construction, we have that
$$
h(z)=h(z)^*,\quad \text{for}\quad z\in \T^d=\{w\in \C^d\colon |w_1|=\ldots=|w_d|=1\}\subset (\C^*)^d.
$$
For $\theta=(\theta_1,\ldots,\theta_d)\in \R^d$, define also (with some abuse of notation)
\beq
\label{eq_def_htheta}
h(\theta):=h(e^{2\pi i\theta_1},\ldots,e^{2\pi i \theta_d}).
\eeq
This way, the torus $\T^d$ can also be identified with $[0,1)^d$. Since the action of $\Z^d$ on $\mathcal V$ is free, one can identify $\mathcal V$ with $\Lambda\times \Z^d$. Let
$$
\mathcal F\colon \ell^2(\Lambda\times \Z^d)\to \rL^2(\Lambda\times \T^d)=\rL^2([0,1)^d;\ell^2(\Lambda))=\rL^2([0,1)^d;\C^N)
$$
be the Fourier transform, defined on vectors with finite support by
$$
(\F\psi)(x,\theta)=\sum_{n\in \Z^d}e^{2\pi i n\cdot\theta}\psi(x,n),\quad \text{where}\quad x\in \Lambda,\,\, \theta\in \T^d,
$$
and extended into $\ell^2$ by continuity. The following proposition is usually known as the direct integral representation in Floquet theory and, in the stated form, can easily be verified by direct calculation.
\begin{prop}
\label{prop_floquet}
The operator $\mathcal F H\mathcal F^{-1}$, acting on $\rL^2([0,1)^d;\ell^2(\Lambda))$, is the operator of multiplication by the matrix-valued function $\theta\mapsto h(\theta)$ defined in \eqref{eq_hz_def}, \eqref{eq_def_htheta}.
\end{prop}
Suppose now that $H$ is self-adjoint, that is, \eqref{eq_self_adjoint} holds. For $\theta\in [0,1)^d$, let $E_j(\theta)$ be the $j$-th eigenvalue of $h(\theta)$, counting multiplicity, in the non-decreasing order. The $j$-th spectral band
$$
[E_j^-,E_j^+]:=\{E_j(\theta)\colon \theta\in [0,1)^d\}
$$
is defined as the range of the corresponding band function. Proposition \ref{prop_floquet}, implies
\beq
\label{eq_spectral_types}
\sigma(H)=\cup_{j=1}^N[E_j^-,E_j^+]; \quad \sigma_{\mathrm{ac}}(H)=\bigcup_{j\colon E_j^+>E_j^-}[E_j^-,E_j^+].
\eeq
where $N=\dim\ell^2(\Lambda)$ is the cardinality of $\Lambda$.

As mentioned above, it is possible to have $E_j^-=E_j^+$. In this case, one of the band functions must be constant, which corresponds to an eigenvalue of infinite multiplicity for $H$. It is natural to call this situation a {\it flat band}. The general definition of a flat band requires some additional considerations, since an eigenvalue of $h(z)$ that is constant in $z$ may be split between different functions $E_j$ due to band crossings that would change eigenvalue numeration. Additionally, while we are mostly interested in the self-adjoint case 
\eqref{eq_self_adjoint}, some of the arguments to follow require involving complex-valued potentials, in which case eigenvalue ordering does not make sense. In order to address both of these issues, we will be using the following well-known result.
\begin{prop}
\label{prop_flat_equivalent}
Define $H$ and $h(z)$ as above. For $E\in \C$, the following are equivalent:
\begin{enumerate}
    \item $E$ is an eigenvalue of $H$.
    \item $E$ is an eigenvalue of $H$ of infinite multiplicity.
    \item For some $j$, the set $\{\theta\in [0,1)^d\colon E_j(\theta)=E\}$ has positive Lebesgue measure in $[0,1)^d$.
    \item $E$ is an eigenvalue of $h(z)$ for all $z\in (\C^*)^d$.
\end{enumerate}
\end{prop}
The equivalence between (1) and (2) is the consequence of the fact that $\Z^d$-translation of every eigenvector is a different eigenvector with the same eigenvalue, the equivalence between (1) and (3) is the general property of multiplication operators, and the equivalence between (3) and (4) is the consequence of the analytic continuation principle for the functions $z\mapsto \det (h(z)-E)$ and standard arguments from measure theory. We refer the reader to \cite{Overview,Ku_book} for a more detailed review of Floquet theory. More specifically, see also \cite{RS4,fs} for the discussion regarding direct integrals and \cite{Sabri,FLPreprint} for some calculations specific for the discrete case.
 
We say that the operator $H$ (not necessarily self-adjoint) has a {\it flat band} at the energy $E$ if one of the claims in Proposition \ref{prop_flat_equivalent} holds.
\begin{remark}
\label{rem_sabri}
One can also show (see, for example, \cite{Sabri}) that the eigenspace associated to a flat band will always be spanned by compactly supported eigenfunctions, although we will not be using this fact.
\end{remark}

\subsection{The main result} Since the potential $\mathfrak v$ is $\Z^d$-periodic, it is uniquely determined by its values on $\Lambda$. As above, let us identify $\Lambda$ with $\{1,\ldots,N\}$, where $N=\#\Lambda$, and let $(V_1,\ldots,V_N)\in \C^N$ be the corresponding values of the potential.
\begin{theorem}
\label{th_main_graph}
Fix $\mathcal V$, $\tau$, and $\mathfrak a$ as in $(p1)$ -- $(p4)$ and suppose that the associated graph $\Gamma$ is connected. Then the set
\beq
\label{eq_v_set}
\{(V_1,\ldots,V_N)\in \C^N\colon H\text{ it has a flat band at some energy $E\in \C$}\}
\eeq
is contained in some proper affine algebraic sub-variety of $\C^N$.
\end{theorem}
\begin{remark}
\label{rem_direct}
In a more direct language, the set \eqref{eq_v_set} is the intersection of finitely many zero sets of non-constant polynomials in $V_1,\ldots,V_N$ with complex coefficients (the number of the polynominals depends on the cardinality of $\Lambda$ and the degrees of vertices in the graph $\Gamma$). Here, we use the terminology of \cite{CLOIVA}. In some other textbooks, subsets of this kind would referred to as closed affine algebraic subsets of $\C^N$, reserving the notion of an algebraic variety for objects with more structure.

As a consequence, the complement of \eqref{eq_v_set} in $\C^N$ contains a Zariski open subset, which is also an open dense connected subset in the usual (analytic) topology. Moreover, the set of real potentials in \eqref{eq_v_set} is also contained in a proper algebraic sub-variety of $\R^N$ (defined by the same equations), and therefore the set of real potentials that do not have flat bands also contains an open and dense subset of $\R^N$.

In the self-adjoint case, one can easily check that flat bands are only possible for $E\in\R$.
\end{remark}
\begin{remark}
\label{rem_connected}
If $(p1)$ -- $(p4)$ are satisfied but $\Gamma$ is not connected, then one can still apply Theorem \ref{th_main_graph} to each infinite connected component of $\Gamma$. The associated operators $H$, acting on smaller subspaces, will also have no flat bands generically.

The operator associated to each finite connected component, considered together with its translations by the $\Z^d$-action, will satisfy $(p1)$ -- $(p4)$, but will only have eigenvalues of infinite multiplicities in the spectrum, producing a trivial case of flat bands that was mentioned earlier in the Introduction. 

As a consequence, the assumption of the main theorem can be relaxed to $\Gamma$ not having any finite connected components.
\end{remark}

\subsection{Discussion and references} As mentioned in the introduction, for general connected graphs one cannot expect flat bands for all potentials; see, for example, \cite[Sections 3 and 4]{Sabri} for a large class of examples. The well-known example of Lieb lattice is also described in Subsection 3.4 of the present paper. Our main result, Theorem \ref{th_main_graph}, answers affirmatively the question \cite[Problem 2]{Sabri}. In this setting, edge weights are fixed and the potential is generic. In the case where both edge weights and the potential are allowed to be generic, a version of this question has been affirmatively answered in the recent preprint \cite{FLPreprint}. However, except for the setting and some general considerations, there is very little similarity between the method of \cite{FLPreprint} and the present paper. The work on the present paper started after \cite{FLPreprint} appeared.

There is a large body of work regarding absence of eigenvalues in the spectra of periodic operators in the continuum, starting from the classical result of \cite{Thomas}, without any genericity assumptions. A comprehensive list of references would probably span over 100 publications. We refer the reader to \cite{BSu99} and \cite[Section 6]{Overview} for some of them, see also \cite[Chapter 16]{RS4} and \cite{Ku_book} for a more textbook-like treatment. In some particular cases of discrete graphs, such as $\Z^d$ with a maximal rank lattice of periods, arguments along the lines of \cite{Thomas} also imply absence of eigenvalues/flat bands for discrete Schr\"odinger operators; however, it is harder to track exact original references, since the proofs are quite elementary and were likely to be known in the community for a while. See, for example, \cite[Theorem 2.6]{Kruger} or \cite[Proposition 3.1]{fk2}, although neither of the references claims it as an original result. See also \cite{Higuchi,Sabri,SK,ktw}.

More generally, the presence of a flat band indicates that the Bloch variety associated to the corresponding operator is reducible. Therefore, absence of flat bands can be concluded from irreducibility of the said variety, which has lately become a topic of a large body of work.  Some examples of recent results in this area are contained in \cite{Overview, shipmansottile, liujmp22, kuchment2023analytic,fg25,flm22,liu1,lslmp20,flm23,LeeLiShip, Shipman}, along with applications \cite{LiuQE,Shipman2, Embed, lmt, liu1,Kuchment1998,KV, KV22}.

\subsection{Structure of the proof} Due to the complex-analytic nature of the problem, it is reasonable to expect that absence of flat bands is an ``analytic condition'': that is, if one produces a non-empty open subset of potentials with this property in $\C^N$, then it would automatically hold for an open dense subset. Since the algebraic formulation of this condition requires an elimination of a quantifier, this is not completely obvious, however, it is made precise in Corollary \ref{cor_generic_flat_bands} in the next section.

As a consequence, it remains to show that an open subset of potentials does not have flat bands. The luxury of choosing an open subset allows us to consider a perturbative regime where exact calculations are possible, which is the case where all values $V_s$ are distinct, large, and are separated from one another in $\C^N$. By rescaling, one can instead fix any potential with distinct values and put a small parameter $\varepsilon$ in front of the hopping term  (small coupling regime). In this case, the eigenvalues of the rescaled operator $h_{\ep}(z)$, assuming that the components of $z$ are away from $0$ and $\infty$, are small perturbations of the diagonal entries of $h_{\ep}(z)$, which are the values of the potential. Perturbation series for these eigenvalues, known as {Rayleigh--Schr\"odinger series}, are well-known in physics literature. 

In the setting of fixed edge weights and generic $V$, the problem essentially reduces to showing that one cannot have complete cancellations among certain classes of the terms in these series. While this is now a purely algebraic problem, the general structure of the graph and the use of diagrams makes it surprisingly non-trivial.

In Section 2.1, we state the rescaled version of the problem and the general ``soft'' result of the analytic perturbation theory (Proposition \ref{prop_analytic_branches}). In Subsection 2.2, we describe the structure of perturbation series in the small hopping regime and describe in detail the language of the associated diagrams (which we call loop configurations). We are using the modified notation of \cite{KPS}, however, we note that multiple different forms of these series have appeared in literature over the last 100 years (see also \cite{Arnold} for some discussions). As mentioned in Subsection 2.4, in the regime of the present paper there are no issues with convergence of the series and ordering of the terms. In Subsection 2.5, which concludes with Corollary \ref{cor_main_reduction}, we discuss several reductions of the original question in Theorem \ref{th_main_graph} to, ultimately, a question about possible cancellations among certain subsets of terms in the series.

These cancellations are addressed in Section 3. The main idea is to find a class of terms, described above, that contains only one non-trivial loop configuration. This class is described in Subsection 3.1 and is obtained by maximizing and minimizing certain combinatorial properties of the associated path, in a particular order. The main result of Section 3 is Theorem \ref{th_main_loops}, which shows that such extremal configuration does indeed produce a term that does not cancel for generic $V$. The proof of Theorem \ref{th_main_loops} spans Subsections 3.2 and 3.3. For graphs with edges of a certain kind, the original definition of an extremal loop configuration is not sufficient, and one has to go (very carefully) to the next order of perturbation theory to observe the absence of cancelations. Some examples are considered in the final Subsection 3.4.

\section{Perturbation theory in the regime of small coupling}
The goal of this section is to produce, under the assumptions of Theorem \ref{th_main_graph}, a set of values of the potentials that would guarantee no flat bands in the sense of Proposition \ref{prop_flat_equivalent}. It will be convenient to introduce some rescaling and put a small parameter in front of the second term.

\subsection{The setting} Let $\mathcal A\subset \Z^d$ be a finite set and $N\in \N$. Let also
$$
V:=\diag\{V_1,\ldots,V_N\},\quad z:=(z_1,\ldots,z_d),\quad z^{\alpha}:=z_1^{\alpha_1}z_2^{\alpha_2}\ldots z_d^{\alpha_d}.
$$
Finally, let
\beq
\label{eq_delta_def}
\{{\mathfrak b}_{\alpha}\colon \alpha\in\A\},\quad ({\mathfrak b}_{\alpha})_{ij}\neq 0 \,\text{ if and only if }\,({\mathfrak b}_{-\alpha})_{ji}\neq 0,\quad ({\mathfrak b}_0)_{ii}=0.
\eeq
be a family of $(N\times N)$-matrices indexed by $\alpha\in \A$. We will considering the following matrix family on $\C^N$:
\beq
\label{eq_he_def}
h_{\ep}(z):=V+\ep\sum\limits_{\alpha\in \A}z^\alpha {\mathfrak b}_{\alpha},
\eeq
where $\ep$ is a small parameter.

For $r>0$, let
\beq
\label{eq_Br_def}
\mathcal B_r:=\{(V_1,\ldots,V_N)\in \C^N\colon \min\{|V_i-V_j|\colon 1\le i<j\le N\}>r\}
\eeq
be the space of potentials whose values are separated from one another with distances bounded from below by $r$. Clearly, $\mathcal V_r$ is an open subset of $\C^N$.

Suppose also that the values of $z$ are restricted to an open subset $Z\subset (\C^*)^d=\l(\C\setminus\{0\}\r)^d$, whose closure is compact in $(\C^*)^d$. Finally, let
\beq
\label{eq_mb_def}
M_{{\mathfrak b}}:=\max\{\|{\mathfrak b}_{\alpha}\|\colon \alpha\in \A\}
\eeq
be a bound on the norms of the terms of \eqref{eq_h_def}. The following result is well known.
\begin{prop}
\label{prop_analytic_branches}
Fix $r,M_{{\mathfrak b}}>0$ and some $Z\subset (\C^*)^d$ as above. There exists $\ep_0=\ep_0(r,\A,Z,M_{{\mathfrak b}})>0$ such that the operator \eqref{eq_h_def} has a family of eigenvalues
$$
\lambda_j=\lambda_j(\ep),\quad \lambda_j(0)=V_j,\quad j=1,\ldots,N
$$
which are analytic in $\{\ep\in \C\colon |\ep|<\ep_0\}$. Moreover, if one considers them as functions of $z$ and $V$, they are also analytic in those variables assuming $V\in \mathcal B_r$ and $z\in Z$.
\end{prop}
\begin{remark}
\label{rem_multivariate}
Note that one has to be careful in defining multivariate analytic functions, especially in subsets that are not necessarily simply connected. However, we will soon provide non-ambiguous explicit expressions for these eigenvalue branches that are valid in the whole range under consideration (perhaps after modifying the choice of $\ep_0$).
\end{remark}

\subsection{Perturbation series}
As a consequence of Proposition \ref{prop_analytic_branches}, the eigenvalues $\lambda_j$ can be expressed as formal power series in $\ep$, converging for $|\ep|<\ep_0$. These series are commonly known as Rayleigh--Schr\"odinger perturbation series. While the total coefficient at $\ep^k$ in the series for $\lambda_j$ is a unique well-defined object, the complete expression for this object is usually combinatorial and involves some choices that one can make in order to obtain a convenient representation. Due to the complexity of the expansion, the notation also plays an important role. We will use the construction described in \cite{KPS} with some modifications that are specific to the structure of the operator family \eqref{eq_h_def}.

The first object that we will need to consider is a simple loop, which is an expression of the form
\beq
\label{eq_p_def}
\cP=n_0\xrightarrow{\alpha_1}n_1\xrightarrow{\alpha_2}\ldots\xrightarrow{\alpha_k}n_k,
\eeq
where
$$
n_0=n_k=j\in \{1,\ldots,N\};\quad  n_1,\ldots n_{k-1}\in\{1,\ldots, N\}\setminus \{j\};\quad \alpha_i\in \A.
$$
One can consider $\cP$ as a representation of a closed path on $\{1,\ldots,N\}$, starting and ending at $j$ and not being allowed to visit $j$ in between. On each step, one is allowed to choose which of the operators ${\mathfrak b}_{\alpha}$ (``hopping terms'') will be used to perform the step. It is convenient to consider an associated oriented multi-graph $\mathcal G$ whose vertices are $\{1,\ldots,N\}$, and each non-zero matrix element $({\mathfrak b}_{\alpha})_{ij}$ is associated to an edge between $i$ and $j$. We will say that this edge has {\it quasimomentum} $\alpha$ to distinguish contributions from different ${\mathfrak b}_{\alpha}$.

As mentioned above, \eqref{eq_p_def} is called a {\it simple loop}, and it will give a contribution to a coefficient at $\lambda_j$. We will sometimes refer to it as a (simple) $j$-loop, in order to emphasize the fact that $n_0=n_k=j$. In order to specify the value of this contribution, let us introduce some notation. For $\cP$ defined as in \eqref{eq_p_def}, let
$$
\cont(\cP):=({\mathfrak b}_{\alpha_1})_{n_0 n_1}(V_j-V_{n_1})^{-1}({\mathfrak b}_{\alpha_2})_{n_1 n_2}(V_j-V_{n_2})^{-1}({\mathfrak b}_{\alpha3})_{n_2 n_3}\ldots (V_j-V_{n_{k-1}})^{-1}({\mathfrak b}_{\alpha_{k}})_{n_{k-1} n_k}.
$$
One can interpret this contribution as a product of ``vertex factors'' $(V_j-V_{n})^{-1}$ obtained by visiting a point $n\in \{1,\ldots N\}\setminus \{j\}$, and ``edge factors'' $({\mathfrak b}_{\alpha_s})_{n_{s-1}n_s}$ that are generated by the hopping process. For $\cP$ as in \eqref{eq_p_def}, define it's {\it length} and {\it quasimomentum}, and the {\it footprint} by
$$
|\cP|:=k,\quad \q(\cP):=\alpha_1+\ldots+\alpha_k,\quad \foot(\cP)=\{n_1,\ldots,n_{k-1}\}_{\mult}
$$
to be the number of edges, the total quasimomentum of the edges, and the multiset of vertices visited by $\cP$. In $\foot(\cP)$, each vertex is counted as many times as it appears, but the order does not matter.
The complete contribution by $\cP$ into $\lambda_j$ will be $\ep^k z^{\q(\cP)}\cont(\cP)$.

Unfortunately, this does not end the combinatorial construction. In order to obtain the complete description of $\lambda_j$, we also need to describe the attachment process. Suppose that $\cP$ is a simple $j$-loop as in \eqref{eq_p_def}. The attachment procedure is replacing any entry $n_s$, $s=1,\ldots,k-1$, by an expression
$$
n_s\xrightarrow{}(m_0\xrightarrow{\beta_1}m_1\xrightarrow{\beta_2}\ldots\xrightarrow{\beta_\ell}m_\ell)\xrightarrow{} n_s,
$$
where $m_0\xrightarrow{\beta_1}m_1\xrightarrow{\beta_2}\ldots\xrightarrow{\beta_\ell}m_\ell$ is another $j$-loop. This construction can be iterated: in the above example one can replace any $m_i\neq j$ by any $j$-loop, and so on. The result of finitely many such procedures is called a {\it loop configuration}.

Note that the entry $n_s$ became duplicated, which means that each of the copies can be used for further iterating the attachment. However, the arrows that lead to the inserted loop have no indexes over them, which reflects the fact that they will not contribute further to the quasimomentum. Each attachment procedure will generate an extra factor $-(V_{j}-V_{n_s})^{-1}$. More precisely, let 
$$
\cP=n_0\xrightarrow{\alpha_1}n_1\xrightarrow{\alpha_2}\ldots\xrightarrow{\alpha_k}n_k,\quad \cQ=m_0\xrightarrow{\beta_1}m_1\xrightarrow{\beta_2}\ldots\xrightarrow{\beta_\ell}m_\ell,
$$
where
$$n_0=n_k=m_0=m_\ell=j;\quad n_1,\ldots,n_{k-1},m_1,\ldots,m_{\ell-1}\in \{1,\ldots,N\}\setminus\{j\},
$$
and $\cP'$ be the result of attaching $\cQ$ to $\cP$ in the location $n_s$:
\beq
\label{eq_pprime_def}
\cP'=n_0\xrightarrow{\alpha_1}n_1\xrightarrow{\alpha_2}\ldots \xrightarrow{\alpha_s}n_s\xrightarrow{}(m_0\xrightarrow{\beta_1}m_1\xrightarrow{\beta_2}\ldots\xrightarrow{\beta_\ell}m_\ell)\xrightarrow{}n_s\xrightarrow{\alpha_{s+1}}\ldots \xrightarrow{\alpha_k}n_k.
\eeq
In other words, a copy of $n_s$ in the expression for $\cP$ is replaced by the expression $n_s\xrightarrow{}(m_0\xrightarrow{\beta_1}m_1\xrightarrow{\beta_2}\ldots\xrightarrow{\beta_\ell}m_\ell)\xrightarrow{}n_s$.

We define
$$
\cont(\cP'):=-(V_j-V_{n_s})^{-1}\cont(\cP)\cont(\cQ),\quad |\cP'|:=|\cP|+|\cQ|,
$$
$$
\foot(\cP'):=\foot(\cP)\cup\foot(\cQ)\cup\{n_s\}_{\mult},
$$
where the union is considered in the sense of multisets (that is, counting multiplicities). We also define
$$
\q(\cP'):=\q(\cP)+\q(\cQ)=\alpha_1+\ldots+\alpha_k+\beta_1+\ldots+\beta_\ell.
$$
As mentioned above, the attachment procedure can be iterated, with the same rules applied inductively. For example, in \eqref{eq_pprime_def}, further attachments can be performed at $n_1,\ldots,n_{k-1}$, $m_1,\ldots,m_{\ell-1}$. Since there are two copies of $n_s$, the attachments can be performed at each of them separately, in general, producing a different loop configuration (unless the attachment is of the same loop $\cQ$ as before).

The following proposition is the complete characterization of the Rayleigh--Schr\"odinger perturbation series. We refer the reader to \cite{KPS} for the proof, which has slightly different notation but the conversion is straightforward. Note that understanding the notation is perhaps the most challenging aspect of the proof.
\begin{prop}
\label{prop_rayleigh_series}
Let $\lambda_j$ be defined as in Proposition $\ref{prop_analytic_branches}$. Then
\beq
\label{eq_formal_series}
\lambda_j=V_j+\sum_{k=1}^{\infty} \ep^k \sum\limits_{\cP\colon |\cP|=k}z^{\q(\cP)}\cont(\cP),
\eeq
where the inner summation is considered over all loop configurations $\cP$ with $|\cP|=k$.
\end{prop}
\begin{remark}
We also refer the reader to \cite{KPS} for a more detailed and formal description of the attachment procedure. For the purpose of this paper, attachments will not play a significant role, since Step 1 of the main technical result \ref{th_main_loops}, essentially, states that the loop configurations of interest for the purposes of that theorem, will be found among those without attachments.
\end{remark}
\subsection{Relation between the graphs $\mathcal G$ and $\Gamma$} Recall that our original goal is to consider operators that arise from Theorem \ref{th_main_graph}, which involves a graph $\Gamma$ described in (p1) -- (p4) in the Introduction. The operator \eqref{eq_he_def} is the rescaled version of the original fiber operator \eqref{eq_h_def}. The graph $\mathcal G$, constructed earlier in this section, is essentially the quotient graph: $\mathcal G=\Gamma/\Z^d$, where in the definition of the quotient graph one also has to consider the edges that connect between different copies of the fundamental domain $\Lambda$. Each such edge of the quotient graph will ``remember'' $\Z^d$ distance between the copies of the fundamental domain whose vertices it was connecting through the {\it quasimomentum}, which is an additional $\Z^d$-valued edge weight function.

More precisely, if $x,y\in \Lambda$ and $\Gamma$ has an edge between $x$ and $y+\alpha$ with weight $\mathfrak a(x,y+\alpha)$, then $\mathcal G$ will have an edge between $x$ and $y$ with weight $\mathfrak a(x,y+\alpha)$ and quasimomentum $\alpha$. Instead of considering $\mathcal G$ being a multi-graph, one could also consider edge weights that are elements of $\C[z_1,z_1^{-1},\ldots,z_d,z_d^{-1}]$, making it more in line with the notation in \eqref{eq_b_def}. However, it will be important to consider contributions from different monomials in these edge weights separately.

Note that $\Gamma$ is connected if and only for every $x,y\in \mathcal G$ and every $\alpha\in \Z^d$, there is a path in $\mathcal G$ between $x$ and $y$ with non-zero weight and quasimomentum $\alpha$. In particular, if $\Gamma$ is connected then $\mathcal G$ is also connected.

We will say that $\mathcal G$ is {\it multi-connected} if there exist $\alpha_1, \alpha_2\in \A$, $\alpha_1\neq \alpha_2$, such that $({\mathfrak b}_{\alpha_1})_{ij}\neq 0$ and $({\mathfrak b}_{\alpha_2})_{ij}\neq 0$. In other words, $\mathcal G$ is multi-connected if there exists a pair of vertices that is connected by two edges of different quasimomenta. If such a pair of vertices does not exist, we will call $\mathcal G$ {\it single-connected}.

\subsection{Convergence of the series \eqref{eq_formal_series}}
In preparation for the proof of the main result, we will summarize some claims regarding the domain of convergence of the series \eqref{eq_formal_series}. For a fixed $j$, denote
\beq
\label{eq_W_def}
W_s:=(V_j-V_s)^{-1},\quad s\in \{1,\ldots,N\}\setminus\{j\}.
\eeq
Recall also \eqref{eq_mb_def} defining $M_{{\mathfrak b}}=\max\{\|{\mathfrak b}_{\alpha}\|\colon \alpha\in \A\}$.
\begin{lemma}
\label{lemma_convergence}
There exists $C=C(\Gamma)=C(d,\#\mathcal A,\#\Lambda)$ such that the number of loop configurations of length $k$ grows at most exponentially in $k$:
\beq
\label{eq_cardinality_p}
\#\{\cP\colon |\cP|=k\}\le C(\Gamma)^k.
\eeq
As a consequence, for any open $Z\subset (\C^*)^d$ with compact closure and any $r>0$, there exists $\varepsilon_0=\varepsilon_0(M_{{\mathfrak b}},Z,r,C(\Gamma))$ such that for $|\ep|<\ep_0$ the series \eqref{eq_formal_series} converges uniformly and absolutely, considered as a Laurent series in $\ep$, $W_1,\ldots,W_N$, $z_1\ldots,z_d$ in the region of $\C^{N+d}$ defined by
$$
|\ep|<\ep_0, \quad z\in Z,\quad |W_s|<r,\quad s\in \{1,\ldots,N\}\setminus \{j\}.
$$
Moreover (as a consequence), the same is true for any combination of  derivatives in the said variables.
\end{lemma}
\begin{proof}
From the structure of the series, it is easy to see that all statements essentially follow from \eqref{eq_cardinality_p} (for example, using the Weirstrass' $m$-test). Regarding the estimate of the cardinality, note that every loop configuration can be described by a sequence of $k$ choices, where at each step one can either travel along one of the edges of $\mathcal G$, or start or terminate an attached loop whenever it is allowed by the structure of the part. Clearly, the number of choices at each step is bounded by $C(\Gamma)$.
\end{proof}

\begin{remark}
\label{rem_resummation}
In situations such as \cite{KPS}, it is important to consider the terms at each $\varepsilon^k$ as one expression, since the series may become diverging absolutely if one considers separately the terms corresponding to different monomials in $W_s$ or different loop configurations associated to the same monomial. However, in the present setting the smallness of $\varepsilon$ and separation between the values of the potential $V_s$ ensure that the series converges absolutely to the same expression in any of these interpretations.
\end{remark}

In the interpretation in Lemma \ref{lemma_convergence}, one considers the series as the sum over all possible monomials in $\{W_s\}$ (corresponding to the choice of the footprint of the corresponding loop configuration) and all possible values of the quasimomenta, grouping together the terms with the same footprints and quasimomenta. Note that the power of $\varepsilon$ is always equal to the cardinality of the footprint, considered as a multi-set.

As a consequence, one can rewrite the sum \eqref{eq_formal_series} as
\beq
\label{eq_resummation}
\lambda_j=V_j+\sum_{k=1}^{+\infty}\sum\limits_{\mathfrak f\in \mathcal F_{j,k}}\sum_{\alpha\in \Z^d}\ep^k W^{\mathfrak f}z^{\alpha}\totalcont(\mathfrak f,\alpha),
\eeq
where:
\begin{itemize}
    \item $\mathcal F_{j,k}$ is the set of possible footprints among loop configurations of length $k$ contributing to $\lambda_j$; in other words, sub-multisets of $\{1,\ldots,N\}\setminus\{j\}$ of cardinality $k$.
    \item $W^f$ is the monomial in $\{W_s\}:=\{W_1,\ldots,W_N\}\setminus \{W_j\}$, whose multi-power is equal to $\mathfrak f$.
    \item $\totalcont(\mathfrak f,\alpha)$ is the total contribution of loop configurations with footprint $\mathfrak f$ and quasimomentum $\alpha$, after factoring out $z^{\alpha}$ and $W^{\mathfrak f}$. In other words,
    $$
    \totalcont(\mathfrak f,\alpha)=\prod_{s\in \mathfrak f}(V_j-V_s)\sum_{\foot(\cP)=\mathfrak f,\,\quasi(\cP)=\alpha}\cont(\cP),
    $$
    where the product over the multi-set $\mathfrak f$ is considered counting multiplicity. Note that, similarly to $\cont(\cdot)$, the object $\totalcont$ depends on $j$ through the allowed range of values of $\mathfrak f$ and structures of loop configurations.
    \item For each $\mathfrak f\in \mathcal{F}_{j,k}$, the summation in $\alpha$ is finite. One can check that, possibly after modifying the value of $C(\Gamma)$, one can assume that $|\alpha|\le (C(\Gamma))^k$.
    \item The triple sum converges uniformly and absolutely on the set described in Lemma \ref{lemma_convergence}.
\end{itemize}

\subsection{Sufficient conditions for non-constant band functions} In this section, we will formalize the following two intuitively plausible observations:
\begin{itemize}
    \item For $\varepsilon$ small enough, one cannot have a complete cancellation between two terms in \eqref{eq_resummation} with different values of $k$, unless both terms are equal to $0$.
    \item In the problems involving generic potentials, $\{W_s\}$ can be treated as independent formal variables.
\end{itemize}
\begin{lemma}
\label{lemma_separate_powers}
Under the assumptions of Propositions $\ref{prop_analytic_branches}$, $\ref{prop_rayleigh_series}$, and Lemma $\ref{lemma_convergence}$, fix $\{W_s\}$. There exists $\ep_0=\ep_0(j,\{W_s\},r,\A,Z,M_{{\mathfrak b}})>0$ such that, for $|\ep|<\ep_0$ the following holds: the function $\lambda_j$ defined by the series \eqref{eq_formal_series} is constant in $z$ if and only if the coefficient at each $\ep^k$ is constant in $z$.
\end{lemma}
\begin{proof}
Clearly, if all terms are constant in $z$, then $\lambda_j$ is also constant in $z$.

In order to establish the reverse direction, let $z_0\in (C^*)^d$ and fix a small neighborhood $C\ni z_0$ with compact closure $\overline{C}\subset (C^*)^d$. In view of Lemma \ref{lemma_convergence}, 
one can find a large number $B=B(r,\A,Z,M_{{\mathfrak b}},C)>0$ such that
\beq
\label{eq_compare_derivatives}
\l|\sum\limits_{\cP\colon |\cP|=k}z^{\q(\cP)}\cont(\cP)\r|+\l|\nabla_z\sum\limits_{\cP\colon |\cP|=k}z^{\q(\cP)}\cont(\cP)\r|\le B^k,\quad \forall z\in \overline{C}.
\eeq
In particular, the series converges uniformly with its derivatives in $z$ for $z\in C$.

Suppose that the coefficient at at least one power of $\ep$ in \eqref{eq_formal_series} is not constant, and let $k$ be the smallest power with this property. Since the coefficient is a Laurent polynominal in $z$, we have, for some $z\in C$,
\beq
\label{eq_violate}
\l|\nabla_z\sum\limits_{\cP\colon |\cP|=k}z^{\q(\cP)}\cont(\cP)\r|>c=c(V_1,\ldots,V_N,r,\A,Z,M_{{\mathfrak b}},C)>0.
\eeq
Now, one can choose $\ep_0>0$ satisfying
$$
\ep_0 B^{k+1}+\ep_0^2 B^{k+2}+\ldots<c/2,
$$
guaranteeing that the contribution from the derivatives of the remaining terms of the series would not completely cancel out \eqref{eq_violate}.
\end{proof}
Instead of considering the whole series \eqref{eq_formal_series}, Lemma \ref{lemma_separate_powers} now allows us to look for a non-constant term and then choose a sufficiently small $\ep$, which can later be converted into existence of the potential with no flat bands by rescaling. We will now deal with the second claim from the beginning of the subsection, regarding generic values of $\{W_s\}$.

\begin{lemma}
\label{lemma_fixed_footprint}
Fix $j$ and consider the series \eqref{eq_formal_series}. The following claims hold:
\begin{enumerate}
    \item Suppose that the term $\sum\limits_{\cP\colon |\cP|=k}z^{\q(\cP)}\cont(\cP)$ is constant in $z$ for all choices of $\{W_s\}$. Consider the said term as a polynomial in $\{W_s\}$. Then, the coefficient at each monomial in $\{W_s\}$ is constant in $z$.
    \item The set of values of $\{W_s\}$ such that each term of \eqref{eq_formal_series} is constant in $z$, is an affine algebraic sub-variety of $\C^{N-1}$.
\end{enumerate}
\end{lemma}
\begin{proof}
The above lemma is a restatement of the well-known fact that a Laurent polynomial $$p\in \C[z_1,z_1^{-1},\ldots,z_d,z_d^{-1},\{W_s\}]$$ identically vanishes in $z_1,\ldots,z_d$ if and only if the coefficient at every monomial in $\{W_s\}$ (which is a Laurent polynomial in the variables $z$) identically vanishes. The latter is true for polynomials over any field of characteristic zero. Clearly, the converse is also true.

In the second claim, it is easy to see that identical vanishing of each individual term of \eqref{eq_formal_series} reduces to finitely many algebraic equations in the variables $\{W_s\}$, which makes the set under consideration a countable intersection of affine algebraic varieties. It is well known (see, for example, \cite[Chapter 2, \S 5]{CLOIVA}) that such intersection is still a set of the same kind.
\end{proof}
With the above preparations, we are ready to proceed with two last steps of the reductions: from a non-zero constant coefficient in the series \eqref{eq_resummation} to an open subset of potentials with no flat bands, and from an open subset of potentials to a generic (dense open/Zariski open) subset of potentials.
\begin{lemma}
\label{lemma_make_open}
Suppose that the assumptions of Lemma $\ref{lemma_convergence}$ are satisfied for all $j$. Suppose also that, for each $j$, there exists a multi-set $\mathfrak f_j\in \mathcal{F}_{j,k}$ for some $k\ge 1$ and a quasimomentum $\alpha_j\neq 0$ such that $\totalcont(\mathfrak f_j,\alpha_j)\neq 0$. Then, there exists a non-empty open subset of potentials on which the series \eqref{eq_formal_series}, \eqref{eq_resummation} converge to non-constant functions in $z$.
\end{lemma}
\begin{proof}
For each $j$, the assumptions imply that the coefficient of the series \eqref{eq_resummation} at $\ep^k W^{\mathfrak f_j}$, which is a Laurent polynomial in $z$, is not identically zero. From Claim 1 of Lemma \ref{lemma_fixed_footprint}, it follows that there is at least one choice of values of $\{W_s\}$ such that the whole coefficient of the same series at $\ep^k$ is not identically zero. Note that, a priori, one cannot guarantee that these values will be associated to a potential within the range of convergence of the series. However, Claim 2 of Lemma \ref{lemma_fixed_footprint} implies now that the affine variety of ``bad'' potentials, constructed in that claim, is a proper sub-variety of $\C^{N-1}$. In other words, the ``good'' property now holds for almost every potential.

The above consideration can be applied to each $j$, which will produce an open and dense subset of values $\{W_s\}$, each of which will be associated to some open and dense subset of potentials $(V_1,\ldots,V_N)\in \C^N$. By intersecting these subsets, we obtain an open dense subset of potentials such that the associated values $\{W_s\}$ satisfy the non-constancy properties simultaneously. Since the set $\mathcal B_r$ constructed in \eqref{eq_Br_def} is also open, one can find a potential $(V_1,\ldots,V_N)\in \mathcal B_r$ with the same property.

The above preparations now allow us to apply Lemma \ref{lemma_separate_powers} for each $j$. By choosing the smallest among the values $\ep_0$ from that lemma, we ultimately arrive to a potential for which all $\lambda_j$ are defined by the converging power series \eqref{eq_formal_series}, \eqref{eq_resummation} and are non-constant in $z$. Finally, let us note that, from Lemma \ref{lemma_convergence}, the derivatives in $z$ of the constructed $\lambda_j$ are continuous in $\{W_s\}$ and therefore in $(V_1,\ldots,V_N)$. As a consequence, the same statement is also true in a small neighborhood of the constructed potential.
\end{proof}

The last step will involve some additional considerations from algebraic geometry. The key result is the following proposition from elimination theory. Recall that an affine algebraic variety is the zero set of a finite collection of polynomials, and Zariski closure of any subset $S\subset \C^N$ is the smallest affine algebraic variety containing $S$. As mentioned earlier in the reference to \cite[Chapter 2, \S 5]{CLOIVA}, any intersection of affine algebraic varieties is still an affine algebraic variety, therefore Zariski closure is well-defined.
\begin{prop}
\label{prop_projection_lemma}
Let $\pi\colon \C^{N+1}=\C\times \C^{N}\to \C^N$ be the projection onto the last $N$ coordinates. Let $S\subset \C^{N+1}$ be an affine algebraic variety, and suppose that $\C^N\setminus\pi(S)$ contains an open subset of $\C^N$ (in the standard analytic topology). Then $\pi(S)$ is contained in a proper affine algebraic sub-variety of $\C^N$.
\end{prop}
\begin{proof}
The result follows almost directly from the closure theorem, see \cite[Section 3, \S 2, Theorem 7]{CLOIVA}, note also the reference to Chapter 4 in the proof of the said theorem. Suppose $\pi(S)$ is not contained in any proper affine subvariety of $\C^N$. Then the Zariski closure $\overline{\pi(S)}^{Zar}$ of $\pi(S)$ must be equal to $\C^N$. However, the closure theorem states that there is a proper affine algebraic variety $W\subset \C^N$ satisfying 
$$
\overline{\pi(S)}^{Zar}\setminus W=\C^N\setminus W\subset \pi(S),\quad \text{which implies}\quad W\supset \C^N\setminus \pi(S).
$$
By the assumption, the latter set contains an open subset in the analytic topology, whose Zariski closure must be equal to $\C^N$. Since $W$ is closed, one must also have $W=\C^N$, which contradicts the earlier conclusion that $W$ is a proper subset.
\end{proof}
\begin{cor}
\label{cor_generic_flat_bands}
Under the assumptions of Theorem $\ref{th_main_graph}$, suppose that an open (in the analytic topology) subset of potentials $(V_1,\ldots,V_N)$ does not have flat bands. Then the set of the potentials that have flat bands is contained in a proper algebraic sub-variety of $\C^N$.
\end{cor}
\begin{proof}
Clearly, the set 
$$
\{(E,V_1,\ldots,V_N)\in \C^{N+1}\colon E\text{ is a flat band for }H\}
$$
is an affine algebraic subset of $\C^{N+1}$, since it can be described by the vanishing of all coefficients of $\det(h(z)-E)$ at non-constant monomials, which leads to finitely many polynomial equations involving $E$ and $V_1,\ldots,V_N$. The rest follows from Proposition \ref{prop_projection_lemma}.
\end{proof}
We finish the section with its main result, which is the combination of two implications stated in Lemma \ref{lemma_make_open} and Corollary \ref{cor_generic_flat_bands}.
\begin{cor}
\label{cor_main_reduction}
Under the assumptions of Theorem $\ref{th_main_graph}$, suppose that, for each $j$, there exists a multi-set $\mathfrak f_j\in \mathcal{F}_{j,k}$ for some $k\ge 1$ and a quasimomentum $\alpha_j\neq 0$ such that $\totalcont(\mathfrak f_j,\alpha_j)\neq 0$. Then the conclusion of $\ref{th_main_graph}$ holds.    
\end{cor}

\section{Combinatorics of extremal loop configurations}
\subsection{Extremal loop configurations}
Corollary \ref{cor_main_reduction} reduces Theorem \ref{th_main_graph} to a purely algebraic statement about absence of cancellations between certain classes of loop configurations. Since we do not have any information about specific values of edge weights, besides the fact that $\Gamma$ is connected, our only hope would be, essentially, to find a footprint $\mathfrak f_j$ and a quasimomentum $\alpha_j\neq 0$ such that there is only one non-zero loop configuration with these parameters. Finding such ``non-cancelable'' loop configuration is a combinatorial problem which is complicated by lack of assumptions on the graph $\Gamma$.

This said loop configuration will be found by declaring several extremal properties, which have to be adjusted based on whether $\mathcal G$ is single-connected or multi-connected (that is, whether or not there are edges between same vertices of $\mathcal G$ with different quasimomenta). A loop configuration $\cP$ will be called  {\it extremal} if:
\begin{enumerate}
	\item $\cont(\cP)\neq 0$.
	\item $\q(\cP)\neq 0$.
	\item $|\cP|$ is smallest possible among loop configurations satisfying the above requirements.
	\item $\foot(\cP)$ has the smallest possible number of distinct elements among loop configurations satisfying the above requirements.
\end{enumerate}
Note that, if there are no loops satisfying (1) and (2), then one would not be able to connect some point on the fundamental domain of $\Gamma$ with any its translations by $\Z^d\setminus\{0\}$ using the edges of $\Gamma$, which would mean that $\Gamma$ is disconnected. Therefore, under the assumptions of Theorem $\ref{th_main_graph}$, the set of loop configurations under consideration is non-empty, and therefore extremal loops exist. Clearly, all such configurations associated to $\lambda_j$ have the same length.
A loop configuration $\cP$ will be called {\it symmetric extremal} if
\begin{enumerate}
	\item $\cont(\cP)\neq 0$.
	\item $\q(\cP)\neq 0$.
	\item $\foot(\cP)$ has exactly one element of multiplicity $1$, and all other elements have multiplicity two.
	\item $|\cP|$ is smallest possible among loop configurations satisfying the above requirements.
\end{enumerate}
At the moment, we do not make any claims regarding existence of such configurations. Finally, we will call a loop configuration $\cP$ {\it non-cancelable} if:
\begin{enumerate}
	\item $\cont(\cP)\neq 0$.
	\item $\q(\cP)\neq 0$.
	\item There is no other loop configuration contributing to $\lambda_j$ with the same quasimomentum and footprint.
\end{enumerate}
The following combinatorial theorem is the main technical result of the paper.
\begin{theorem}
\label{th_main_loops}
Under the assumptions of Theorem $\ref{th_main_graph}$, fix $j$ and suppose that $L$ is the length of (every) extremal loop configuration contributing to $\lambda_j$. Then, either every such extremal loop configuration is non-cancelable, or there exists a non-cancelable symmetric extremal loop of length $L+1$.
\end{theorem}
\subsection{Proof of Theorem \ref{th_main_loops}: the main case}
Let $\cP$ be an extremal loop configuration. Our goal is to perform a series of reductions that will gradually impose more and more requirements on the structure of $\cP$, which, except for one situation, will lead to its uniqueness among loop configurations with the same quasimomentum and footprint. In the exceptional situation, considered in the next subsection, we will be able to construct a non-cancelable symmetric extremal loop by a modification of $\cP$.

{\it 1. No attachments.} Any extremal loop configuration must, in fact, be a simple loop. Indeed, if it contains an attached loop of non-zero quasimomentum, then one can replace $\cP$ by that attached loop, contradicting the minimality of the length. If it contains an attached loop of zero quasimomentum, that attachment can be removed, again producing a shorter loop.

{\it 2. No triple repetitions.} Suppose that $\cP$ visits a vertex $n\in \{1,\ldots,N\}\setminus\{j\}$ three times. With a slight abuse of notation, one can write 
$$
\cP=\cP_1\xrightarrow{\gamma_1} n\xrightarrow{\gamma_2} \cP_2\xrightarrow{\gamma_3} n\xrightarrow{\gamma_4}\cP_3\xrightarrow{\gamma_5}n\xrightarrow{\gamma_6}\cP_4.
$$
Since $\cP$ was the shortest loop, we have $\quasi(\cP_1\xrightarrow{\gamma_1} n\xrightarrow{\gamma_6} \cP_4)=0$. As a consequence, either $\quasi(n\xrightarrow{\gamma_2} \cP_2\xrightarrow{\gamma_3} n)\neq 0$ or $\quasi(n\xrightarrow{\gamma_4} \cP_3\xrightarrow{\gamma_5} n)\neq 0$, and therefore either $\cP_1\xrightarrow{\gamma_1} n\xrightarrow{\gamma_2} \cP_2\xrightarrow{\gamma_3} n\xrightarrow{\gamma_6}\cP_4$ or $\cP_1\xrightarrow{\gamma_1} n\xrightarrow{\gamma_4} \cP_3\xrightarrow{\gamma_5} n\xrightarrow{\gamma_6}\cP_4$ is shorter than $\cP$ and has non-zero quasimomentum.

{\it 3. Symmetry of double repetitions. }Suppose that $\cP$ visits a vertex $n$ twice, so that
$$
\cP=\cP_1\xrightarrow{\gamma_1} n\xrightarrow{\gamma_2} \cP_2\xrightarrow{\gamma_3} n\xrightarrow{\gamma_4}\cP_3.
$$
Since $\cP$ is extremal, we must have $\quasi(\cP_1\xrightarrow{\gamma_1} n\xrightarrow{\gamma_4} \cP_3)=0$. As a consequence, 
\beq
\label{eq_symmetry_quasi}
\quasi(\cP_1\xrightarrow{\gamma_1}n)=-\quasi(n\xrightarrow{\gamma_4}\cP_3).
\eeq
We claim that $\cP_1$ and $\cP_3$ must have the same length. Indeed, suppose that, say, $\cP_1$ is shorter. Then the loop
$$
\cP_1\xrightarrow{\gamma_1} n\xrightarrow{\gamma_2} \cP_2\xrightarrow{\gamma_3} n\xrightarrow{-\gamma_1}{\cP_1^{-1}},
$$
where ${\cP_1^{-1}}$ is $\cP_1$ taken in the reverse direction, has the same quasimomentum as $\cP$ and is shorter. By contradiction, this implies that all repetitions in $\cP$ have to happen symmetrically about the center.

{\it 4. Symmetrization. }Suppose $\cP$ has two repeated vertices $m$ and $n$, and $m$ is visited first: 
$$
\cP=\cP_1\xrightarrow{\gamma_1} m\xrightarrow{\gamma_2} \cP_2\xrightarrow{\gamma_3} n\xrightarrow{\gamma_4}\cP_3\xrightarrow{\gamma_5}n\xrightarrow{\gamma_6}\cP_4\xrightarrow{\gamma_7}m\xrightarrow{\gamma_8}\cP_5.
$$
From \eqref{eq_symmetry_quasi}, we have
$$
\quasi(m\xrightarrow{\gamma_2}\cP_2\xrightarrow{\gamma_3}n)=-\quasi(n\xrightarrow{\gamma_6}\cP_4\xrightarrow{\gamma_7}m).
$$
Therefore, one can replace one of these segments by the other in the reverse order and obtain a new loop
$$
\cQ=\cP_1\xrightarrow{\gamma_1} m\xrightarrow{\gamma_2} \cP_2\xrightarrow{\gamma_3} n\xrightarrow{\gamma_4}\cP_3\xrightarrow{\gamma_5}n\xrightarrow{-\gamma_3}\cP_2^{-1}\xrightarrow{-\gamma_2}m\xrightarrow{\gamma_8}\cP_5
$$
with $\quasi(\cQ)=\quasi(\cP)$. However, unless $\cP_2$ and $\cP_4$ have the same footprints, $\cQ$ will have more repetitions than $\cP$, which contradicts the extremality of $\cP$. As a consequence, the vertices of $\cP_2$ must repeat the vertices in $\cP_4$. 

In view of the previous step, the repetition has to happen in the same order. Using \eqref{eq_symmetry_quasi}, we can conclude that the mirror image structure also applies to the quasimomenta over the arrows in $\cP_2$ and $\cP_4$ (in particular, $\gamma_6=-\gamma_3$ and $\gamma_7=-\gamma_2$). Summarizing the steps so far, we have that every extremal loop configuration has to be of the form
\beq
\label{eq_p_form}
\cP=n_0\xrightarrow{\alpha_1}n_1\xrightarrow{\alpha_2}\ldots\xrightarrow{\alpha_s}n_s\xrightarrow{\beta_1}m_1\xrightarrow{\beta_2}\ldots\xrightarrow{\beta_\ell}m_\ell\xrightarrow{\beta_{\ell+1}}n_s\xrightarrow{-\alpha_s}n_{s-1}\xrightarrow{-\alpha_{s-1}}\ldots\xrightarrow{-\alpha_1}n_0,
\eeq
where the vertices denoted by different letters must be distinct.

{\it 5. Uniqueness of the repeated part. }As follows from Part 4 and the extremality of $\cP$, the values $\alpha_1,\ldots,\alpha_s$ in \eqref{eq_p_form} are uniquely determined by the vertices $n_1,\ldots,n_s$ and the order in which they are taken. We now claim that they are, in fact, determined by (the repeating part of) the footprint of $\cP$. Indeed, suppose that $\cQ$ is another extremal loop of the form \eqref{eq_p_form}, but with $n_1,\ldots,n_s$ taken in a different order (since both are $j$-loops, the first and last vertex of $\cP$ and $\cQ$ must be equal to $j$). Suppose that 
$$
\cQ=\cQ_1\xrightarrow{\gamma} n_i \xrightarrow{}\ldots \xrightarrow{} n_i\xrightarrow{-\gamma}\cQ_1^{-1},
$$
where $n_i$ is the first vertex of $\cQ$ that is different from the corresponding one in $\cP$, so that $\cQ_1$ is some prefix of $\cP$ and $\cQ_1^{-1}$ is its mirror image. Then the loop
$$
\cQ_1\xrightarrow{\gamma} n_i\xrightarrow{\alpha_{i+1}}\ldots\xrightarrow{\alpha_s}n_s\xrightarrow{\beta_1}m_1\xrightarrow{\beta_2}\ldots\xrightarrow{\beta_\ell}m_\ell\xrightarrow{\beta_{\ell+1}}n_s\xrightarrow{-\alpha_s}\ldots\xrightarrow{-\alpha_{i+1}}n_i\xrightarrow{-\gamma}\cQ_1^{-1}
$$
is shorter than $\cP$, since the repeated part ``skips'' the vertices between $n_i$ and those contained in $\cQ_1$. It also has the same quasimomentum, since the non-repeated part is the same as in $\cP$ and the repeated one cancels out in all cases.

{\it 6. Quasimomenta between adjacent vertices in the non-repeated part. }Once we have determined that there is no ambiguity in the repeated vertices and the corresponding quasimomenta, our attention is now with the non-repeated part
$$
\cQ:=n_s\xrightarrow{\beta_1}m_1\xrightarrow{\beta_2}\ldots\xrightarrow{\beta_\ell}m_\ell\xrightarrow{\beta_{\ell+1}}n_s,
$$
where it is convenient to include $n_s$ into consideration, making it an $n_s$-loop. Before looking into possible ambiguities in the order of vertices $m_i$, let us first assume that the order is fixed and discuss ambiguity between the choices of the values $\beta_i$. Clearly, if $\Gamma$ is single-connected, there is no such ambiguity and this step is not necessary. Suppose that, say, $\beta_i\neq \beta_i'$ are two possible values of quasimomenta between $m_{i-1}$ and $m_i$, and $i<\ell/2+1$. Then, the path
$$
n_s\xrightarrow{\beta_1}m_1\xrightarrow{\beta_2}\ldots\xrightarrow{\beta_i}m_i\xrightarrow{-\beta_{i}'}m_{i-1}
\xrightarrow{-\beta_{i-1}}m_{i-2}\xrightarrow{-\beta_{i-3}}\ldots\xrightarrow{-\beta_2}m_1\xrightarrow{-\beta_1}
n_s
$$
is either shorter than $\cQ$ or of the same length and has more repetitions. Since $\beta_i\neq \beta_i'$, its quasimomentum is still non-zero (everything else cancels). By adding back the original symmetric part of $\cP$, we obtain a contradiction with its extremality. A similar construction (with reflecting the right part instead of the left part) works for $i>\ell/2+1$. If $\ell$ is odd, this covers all possible integer values of $i$. If $\ell$ is even and $i=\ell/2+1$ (the edge located exactly in the middle of $\cQ$), the construction is still possible, but we do not obtain an immediate contradiction, since the length of the path will be increased by one. However, the new constructed path will have only one non-repeated vertex, making it a symmetric extremal path of length $L+1$, a candidate for the second case described in the statement of the lemma. This special case will be considered during the later steps. The edge with multiple values of quasimomenta will be referred to as a {\it multi-edge}.

{\it 7. Order of vertices in the repeated part: special permutation property}. Our next goal in this part is to determine how much ambiguity is allowed in the order of $m_1,\ldots,m_\ell$ in the non-repeated part
$$
\cQ:=n_s\xrightarrow{\beta_1}m_1\xrightarrow{\beta_2}\ldots\xrightarrow{\beta_\ell}m_\ell\xrightarrow{\beta_{\ell+1}}n_s.
$$
Suppose that there is another arrangement
$$
\cQ'=n_s\xrightarrow{\beta_1'}m_1'\xrightarrow{\beta_2'}\ldots\xrightarrow{\beta_\ell'}m_\ell'\xrightarrow{\beta_{\ell+1}'}n_s,
$$
where $(m_1',\ldots,m_\ell')$ is some permutation of $(m_1',\ldots,m_\ell')$. Recall that $\q(\cQ)=\q(\cQ')\neq 0$, otherwise there is nothing to prove. We will show that the corresponding permutation is of the kind considered in Lemma \ref{lemma_special_permutation} below. Indeed, suppose that $\cQ$ contains two points $m,n$ that are in the opposite order in $\cQ'$, so that
$$
\cQ=\cQ_1\xrightarrow{\gamma_1}m\xrightarrow{\gamma_2}\cQ_2 \xrightarrow{\gamma_3}n\xrightarrow{\gamma_4}\cQ_3,\quad \cQ'=\cQ_1'\xrightarrow{\gamma'_1}n\xrightarrow{\gamma_2'}\cQ_2' \xrightarrow{\gamma_3'}m\xrightarrow{\gamma'_4}\cQ_3'.
$$
Since we assumed that no shorter path with non-zero quasimomentum exists, both paths
$$
\cQ_1\xrightarrow{\gamma_1}m\xrightarrow{\gamma_4'}\cQ_3',\quad \cQ_1'\xrightarrow{\gamma_1'}n\xrightarrow{\gamma_4}\cQ_3
$$
must have zero quasimomenta (note that each of them is strictly shorter than both $\cQ$ and $\cQ'$). However, this implies that, if one replaces the part $m\xrightarrow{\gamma_2}\cQ_2 \xrightarrow{\gamma_3}n$ in $\cQ$ by  $m\xrightarrow{-\gamma_3'}\cQ_2^{-1} \xrightarrow{-\gamma_2'}n$ (taken from $\cQ'$ in the reverse order), this will change the quasimomentum of $\cQ$ to the one of the opposite sign, and vice versa. Since both $\cQ$ and $\cQ'$ are of the same length and cannot be shortened, such replacement should not change the length, and therefore the lengths of the corresponding segments must coincide. This confirms the assumptions of Lemma \ref{lemma_special_permutation}.

{\it 8. Order of vertices in the repeated part: flipped segment with no multi-edge}. In view of Lemma \ref{lemma_special_permutation}, the vertices in the alternative configuration $\cQ'$ considered in Step 7 are obtained by a sequence of mirror flips of some non-overlapping segments of $\cQ$. In this step, we will consider the case when at least one flipped segment does not contain the multi-edge considered in Step 6. In particular, this includes the case where there is no multi-edges at all such as the case of odd $\ell$. 

It will be convenient to use the following notation
$$
\cQ=\cQ_1\xrightarrow{\gamma_1}a_1\xrightarrow{\eta_2}a_2\xrightarrow{\eta_3}\ldots\xrightarrow{\eta_t}a_t\xrightarrow{\gamma_2}\cQ_2,
$$
$$
\cQ'=\cQ_1'\xrightarrow{\gamma_1'}a_t\xrightarrow{\eta_t'}a_{t-1}\xrightarrow{\eta_{t-1}'}\ldots\xrightarrow{\eta_2'}a_1\xrightarrow{\gamma_2'}\cQ_2',
$$
where $a_1,\ldots,a_t$ are the vertices of a flipped segment. Assuming that it does not contain a multi-edge, Step 6 implies that we must have 
\beq
\label{eq_quasi_cancel_1}
\eta_i'=-\eta_i,\quad  i=2,\ldots,t,
\eeq 
which explains the choice in the notation. The segments $\cQ_1$, $\cQ_2$ may or may not be the same as the segments $\cQ_1'$, $\cQ_2'$, depending on whether there are multiple flipped intervals, but we always have
$$
|\cQ_1|=|\cQ_1'|,\quad |\cQ_2|=|\cQ_2'|.
$$
This allows us to construct two shorter paths, which both must have zero quasimomenta due to extremality of the original path $\cP$ (of which $\cQ$ is the non-repeating part):
\beq
\label{eq_quasi_cancel_2}
\quasi(\cQ_1\xrightarrow{\gamma_1}a_1\xrightarrow{\gamma_2'}\cQ_2')=\quasi(\cQ_1'\xrightarrow{\gamma_1'}a_t\xrightarrow{\gamma_2}\cQ_2)=0.
\eeq
However,
\begin{multline*}
\quasi(\cQ)+\quasi(\cQ')=\quasi(\cQ_1\xrightarrow{\gamma_1}a_1)+\quasi(a_1\xrightarrow{\eta_2}a_2\xrightarrow{\eta_3}\ldots\xrightarrow{\eta_t}a_t)+\quasi(a_t\xrightarrow{\gamma_2}\cQ_2)+\\+
\quasi(\cQ_1'\xrightarrow{\gamma_1'}a_t)+\quasi(a_t\xrightarrow{\eta_t'}a_{t-1}\xrightarrow{\eta_{t-1}'}\ldots\xrightarrow{\eta_2'}a_1
)+\quasi(a_1\xrightarrow{\gamma_2'}\cQ_2')=\\=\quasi(\cQ_1\xrightarrow{\gamma_1}a_1\xrightarrow{\gamma_2'}\cQ_2')+\quasi(\cQ_1'\xrightarrow{\gamma_1'}a_t\xrightarrow{\gamma_2}\cQ_2)+\\+\quasi(a_1\xrightarrow{\eta_2}a_2\xrightarrow{\eta_3}\ldots\xrightarrow{\eta_t}a_t)+\quasi(a_t\xrightarrow{\eta_t'}a_{t-1}\xrightarrow{\eta_{t-1}'}\ldots\xrightarrow{\eta_2'}a_1)=0
\end{multline*}
due to \eqref{eq_quasi_cancel_1}, \eqref{eq_quasi_cancel_2}. On the other hand, we assumed that both $\cQ$ and $\cQ'$ are non-repeating parts of extreme loops with the same quasimomenta, that is, $\quasi(\cQ)=\quasi(\cQ')\neq 0$, thus producing a contradiction.

This completes the proof under the assumption that the flipped segment does not contain a multi-edge, modulo the following lemma which was announced earlier.
\begin{lemma}
\label{lemma_special_permutation}
Let $\sigma\colon \{1,\ldots,N\}\to \{1,\ldots,N\}$ be a permutation with the following property: for every $1\le i<j\le N$ with $\sigma(i)>\sigma(j)$, we have $|\sigma(i)-\sigma(j)|=|i-j|$. Then there exists a partition of $\{1,\ldots,N\}$ into intervals that are invariant under $\sigma$, and on each interval $\sigma$ acts either as identity of reflection.
\end{lemma}
\begin{proof}
Consider the sequence $\sigma(1),\ldots,\sigma(N)$ and split it into decreasing segments. Note that on each decreasing segment it has to decrease exactly by $1$ each step. As a consequence, if $\sigma(1)\neq 1$, then one must have $\sigma(2)=\sigma(1)-1$ and so one, until $1$ is reached. Afterwards, unless the next value is $\sigma(1)+1$, the sequence would have to decrease again until $\sigma(1)+1$ is reached, and so on.
\end{proof}
\subsection{Proof of Theorem \ref{th_main_loops}: the case of a multi-edge} Suppose that $\Gamma$ contains a multi-edge that appears in a way that does not allow one to complete Step 8 in Subsection 3.2 above. Let $\cP$ be (any) extremal loop, and $\cR$ be the symmetric loop produced during Step 6. Since $|\cR|=|\cP|+1$, we have that $\cR$ is actually symmetric extremal, since any symmetric loop configuration of shorter length will have at least two vertices fewer, contradicting the fact that the loop $\cP$ was the shortest within the class used to define extremal loops.

Let $\cR'$ be another loop configuration with
$$
\foot(\cR')=\foot(\cR), \quad \quasi(\cR')=\quasi(\cR).
$$
We will need to retrace some of the previous steps in order to determine how much ambiguity $\cR'$ has within the above constraints and, ultimately, see whether its contribution may cancel that of $\cR$.

{\it Step 1}: $\cR'$ has no attachments, since an argument similar to Step 1 (that is, removing the attachment or removing everything besides the attachment) would produce an extremal loop of length at most $|\cR|-2=|\cR'|-2$, contradicting the fact that extremal loops must have length $|\cP|=|\cR|-1$.

{\it Step 2}: $\cR$ has no triple repetitions already by construction.

{\it Steps 3 and 4}: symmetrization may potentially reduce the length of $\cR'$ by $1$, making it an extremal loop with no non-repeated vertices. This is only possible if $n_s$ has an edge into itself. Since the quasimomentum of that edge is equal to the quasimomentum of the whole loop, it will not have the issue described in Step 6 and will therefore be non-cancelable. In other words, this case is already considered in the framework of Steps 1 -- 8 in the previous Subsection 3.2.

{\it Step 5}: the argument goes same way as in the previous case.

Summarizing, let 
$$
\cR=n_0\xrightarrow{\alpha_1}n_1\xrightarrow{\alpha_2}\ldots\xrightarrow{\alpha_s}n_s\xrightarrow{\beta}m\xrightarrow{-\gamma}n_s\xrightarrow{-\alpha_s}n_{s-1}\xrightarrow{-\alpha_{s-1}}\ldots\xrightarrow{-\alpha_1}n_0
$$
be a symmetric extremal loop. Then any other loop configuration that may potentially cancel $\cR$ must be of the form
$$
\cR'=n_0\xrightarrow{\alpha_1}n_1\xrightarrow{\alpha_2}\ldots\xrightarrow{\alpha_s}n_s\xrightarrow{\beta'}m\xrightarrow{-\gamma'}n_s\xrightarrow{-\alpha_s}n_{s-1}\xrightarrow{-\alpha_{s-1}}\ldots\xrightarrow{-\alpha_1}n_0.
$$
The choice of the signs will be convenient for the following calculation. In other words, the only ambiguity is the quasimomenta that involve the multi-edge between $n_s$ and $m$. If there are two different pairs of quasimomenta satisfying $\beta-\gamma=\beta'-\gamma'$, then one can have a cancellation between the contributions of $\cR$ and $\cR'$. In order to produce a non-cancelable loop, assume in the choice of $\cR$, in addition, that the Euclidean norm $|\beta-\gamma|$ takes the largest possible value. A potential cancellation between two such pairs would imply
$$
\beta-\gamma=\beta'-\gamma',
$$
which implies that $\beta$, $\gamma$, $\beta'$, and $\gamma'$ are vertices of a parallelogram, with one of the sides of length $|\beta-\gamma|$. However, one of the diagonals $|\beta-\gamma'|$ or $|\beta'-\gamma|$ of the parallelogram will have strictly larger length, and therefore the corresponding choice of multi-edges will provide strictly larger Euclidean norm of the total quasimomentum.\,\,\qed
\subsection{Some examples}
We note that there exist graphs where every extremal loop configuration is cancelable. Consider the following graph.

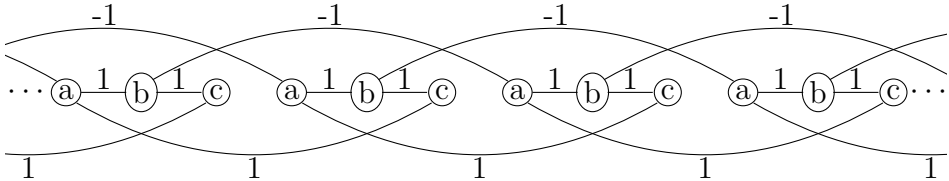
\begin{figure}[ht]
\centering
\begin{tikzpicture}[x=1cm,
        every node/.style={inner sep=1pt},
        edge/.style={-},
        shortenL/.style={shorten >=6pt},   
        shortenR/.style={shorten <=6pt}]   
\clip (-0.8,-2) rectangle (11.7,1.5);

\foreach \i/\lab in {0/a,1/b,2/c,3/a,4/b,5/c,6/a,7/b,8/c,9/a,10/b,11/c}
  \node[ellipse,draw] (\lab\i) at (\i,0) {\lab};

\foreach \i in {0,...,3}{
  \draw[edge] ({3*\i+0.2},0)   -- ++({1-0.4},0) node[midway,above] {1};

  \draw[edge] ({3*\i+1.2},0)   -- ++({1-0.4},0) node[midway,above] {1};}

\foreach \i in {-3,0,3,6,9}
  \draw[edge,bend right] ({\i +0.14},{-0.14}) to node[below] {1} ++({5-0.28},0);

\foreach \i in {-5,-2,1,4,7,10}
  \draw[edge,bend left] ({\i+0.155},{0.17}) to node[above] {-1} ++({5-0.28},{-0.02});

\node at (-0.5,0)  {$\cdots$};
\node at (11.5,0)  {$\cdots$};

\end{tikzpicture}
\end{figure}

The corresponding Floquet transformed operator matrix operator is given by \[h(z) = \begin{pmatrix} V_1 & 1 - z_1^{-2} & z_1 \\ 1 -z_1^2 & V_2 & 1 \\ z_1^{-1} & 1 & V_3 \end{pmatrix}. \]

It is easy to see that with the given edge assignment, the contributions of all loops $\cP$ centered at $c$ with $\q(\cP)\ = \pm 1$ cancel each other. However, the symmetric extremal loops, which are one step longer, do not completely cancel, which should be anticipated from the previous section. As a consequence, the corresponding Schr\"odinger operator does not have any flat bands, but the order of $\ep$ in which one can see that is higher than one would normally expect.


We also note that, 
    for general connected graphs, one should not expect a substantial improvement of the ``generic $V$'' claim. Indeed, there exist graphs where the space of potentials such that they admit a flat band is exactly codimension $1$. Consider, for example, the well studied Lieb lattice, which has the finite matrix operator. 
    \[h(z) = \begin{pmatrix} V_1 & 1 + z_1^{-1} & 1+z_2^{-1} \\ 1 +z_1 & V_2 & 0 \\ 1+z_2 & 0 & V_3 \end{pmatrix}. \]

    It is easy to see that the dispersion relation of this operator has a flat band exactly when $V_2 = V_3$.

\end{document}